\documentclass[11pt]{article}

\usepackage{amssymb,amsmath,amsfonts,amsthm}
\usepackage{latexsym}
\usepackage{graphics}
\usepackage{indentfirst}

\newtheorem{innercustomthm}{Theorem}
\newenvironment{customthm}[1]
  {\renewcommand\theinnercustomthm{#1}\innercustomthm}
  {\endinnercustomthm}

\newtheorem*{thm*}{Theorem}
\newtheorem{thm}{Theorem}
\newtheorem{lem}[thm]{Lemma}

\newtheorem{conj}[thm]{Conjecture}

\newcommand{\N}{\mathbb{N}}

\begin{document}

\title{Total Equitable List Coloring}

\author{Hemanshu Kaul\footnote{Department of Applied Mathematics, Illinois Institute of Technology, Chicago, IL 60616. E-mail: {\tt kaul@math.iit.edu}}, Jeffrey A. Mudrock\footnote{Department of Applied Mathematics, Illinois Institute of Technology, Chicago, IL 60616. E-mail: {\tt jmudrock@hawk.iit.edu}}, and Michael J. Pelsmajer\footnote{Department of Applied Mathematics, Illinois Institute of Technology, Chicago, IL 60616. E-mail: {\tt pelsmajer@iit.edu}}}

%\date{}

\maketitle

\begin{abstract}
An equitable coloring is a proper coloring of a graph such that the sizes of the color classes differ by at most one.  A graph $G$ is equitably $k$-colorable if there exists an equitable coloring of $G$ which uses $k$ colors, each one appearing on either $\lfloor |V(G)|/k \rfloor$ or $\lceil |V(G)|/k \rceil$ vertices of $G$. In 1994, Fu conjectured that for any simple graph $G$, the total graph of $G$, $T(G)$, is equitably $k$-colorable whenever $k \geq \max\{\chi(T(G)), \Delta(G)+2\}$ where $\chi(T(G))$ is the chromatic number of the total graph of $G$ and $\Delta(G)$ is the maximum degree of $G$.
We investigate the list coloring analogue.  List coloring requires each vertex $v$ to be colored from a specified list $L(v)$ of colors.  A graph is $k$-choosable if it has a proper list coloring whenever vertices have lists of size $k$.  A graph is equitably $k$-choosable if it has a proper list coloring whenever vertices have lists of size $k$, where each color is used on at most $\lceil |V(G)|/k \rceil$ vertices.
In the spirit of Fu's conjecture, we conjecture that for any simple graph $G$,  $T(G)$ is equitably $k$-choosable whenever $k \geq \max\{\chi_l(T(G)), \Delta(G)+2\}$ where $\chi_l(T(G))$ is the list chromatic number of $T(G)$.  We prove this conjecture for all graphs satisfying $\Delta(G) \leq 2$ while also studying the related question of the equitable choosability of powers of paths and cycles.
\medskip



\end{abstract}

\section{Introduction}\label{intro}

In this paper all graphs are finite, simple graphs unless otherwise noted.  Generally speaking we follow West~\cite{W01} for terminology and notation.

\subsection{Equitable Coloring and Total Coloring}\label{ec1}

An \emph{equitable $k$-coloring} of a graph $G$ is a proper $k$-coloring of $G$, $f$, such that the sizes of the color classes differ by at most one (where a $k$-coloring has exactly $k$ color classes).  It is easy to see that for an equitable $k$-coloring, the color classes associated with the coloring are each of size $\lceil |V(G)|/k \rceil$ or $\lfloor |V(G)|/k \rfloor$.  We say that a graph $G$ is \emph{equitably $k$-colorable} if there exists an equitable $k$-coloring of $G$.
Equitable coloring has found many applications, see for example~\cite{T73}, \cite{P01}, \cite{KJ06}, and~\cite{JR02}.

Unlike ordinary graph coloring, increasing the number of colors may make equitable coloring more difficult.
Thus, for each graph $G$ we have the \emph{equitable chromatic number} $\chi_=(G)$, the minimum $k$ for which there exists an equitable $k$-coloring, and the \emph{equitable chromatic threshold} $\chi_{=}^*(G)$, the minimum $k$ for which $G$ is equitably $j$-colorable for all $j \geq k$.  For example, $K_{2m+1,2m+1}$ is equitably $k$-colorable for $k\ge 2m+2$ and for even integers $k$ less than $2m+1$, but it is not equitably $(2m+1)$-colorable~\cite{LW96}, so $\chi_{=}^*(K_{2m+1,2m+1})=2m+2$ and $\chi_{=}(K_{2m+1,2m+1})=2$.  It is clear that $\chi(G) \leq \chi_=(G) \leq \chi_{=}^* (G)$ where $\chi(G)$ is the usual chromatic number of $G$.

Erd{\"o}s~\cite{E64} conjectured that $\chi_{=}^*(G) \leq \Delta(G)+1$ for all graphs $G$, where $\Delta(G)$ denotes the maximum degree of $G$.
In 1970, Hajn\'{a}l and Szemer\'{e}di proved it.

\begin{thm}[\cite{HS70}] \label{thm: HS}
Every graph $G$ has an equitable $k$-coloring when $k \geq \Delta(G)+1$.
\end{thm}

In 1994, Chen, Lih, and Wu~\cite{CL94} conjectured an equitable analogue of Brooks' Theorem~\cite{B41}, known as the $\Delta$-Equitable Coloring Conjecture.

\begin{conj}[\cite{CL94}] \label{conj: ECC}
A connected graph $G$ is equitably $\Delta(G)$-colorable if and only if it is different from $K_m$, $C_{2m+1}$, and $K_{2m+1,2m+1}$.
\end{conj}

Conjecture~\ref{conj: ECC} has been proven true for interval graphs, trees, outerplanar graphs, subcubic graphs, and several other classes of graphs~\cite{CL94,L98,YZ97}.

For disconnected graphs, equitable $k$-colorings on components can be merged after appropriately permuting color classes within each component~\cite{YZ97}, to obtain an equitable $k$-coloring of the whole graph.
On the other hand, an equitable $k$-colorable graph may have components that are not equitably $k$-colorable, for example,
the disjoint union $G=K_{3,3}+K_{3,3}$ with $k=\Delta(G)$.
With this in mind, Kierstead and Kostochka~\cite{KK10} extended Conjecture~\ref{conj: ECC} to disconnected graphs.

A \emph{total $k$-coloring} of a graph $G$ is a labeling $f: V(G) \cup E(G) \rightarrow S$ where $|S|=k$ and $f(u) \neq f(v)$ whenever $u$ and $v$ are adjacent or incident in $G$.  For some basic applications of total coloring, see~\cite{L12}.  The \emph{total chromatic number} of a graph $G$, denoted $\chi''(G)$, is the smallest integer $k$ such that $G$ has a total $k$-coloring.  Clearly, for any graph $G$, $\chi''(G) \geq \Delta(G)+1$.  A famous open problem in total coloring is the Total Coloring Conjecture.

\begin{conj}[\cite{B65}] \label{conj: TCC}
For any graph $G$, we have $\chi''(G) \leq \Delta(G)+2$.
\end{conj}

Total coloring can be rephrased in terms of vertex coloring.  Specifically, the \emph{total graph} of graph $G$, $T(G)$, is the graph with vertex set $V(G) \cup E(G)$ and vertices are adjacent in $T(G)$ if and only if the corresponding elements are adjacent or incident in $G$.  Then $G$ has a total $k$-coloring if and only if $T(G)$ has a proper $k$-coloring.  It follows that $\chi''(G) = \chi(T(G))$.

Given a graph $G$, one can construct $T(G)$ in two steps: first subdivide every edge of $G$ to get a new graph $H$, then take its square $H^2$ (i.e. add an edge $uv$ whenever $u,v$ are vertices in $H$ with distance $2$).
For example, for paths on $m$ vertices and cycles on $n$ vertices, we have that: $T(P_m)= P_{2m-1}^2$ and $T(C_n) = C_{2n}^2$.

In 1994, Fu~\cite{F94} initiated the study of equitable total coloring.
A total $k$-coloring of a graph $G$ is an \emph{equitable total $k$-coloring} of $G$ if the sizes of the color classes differ by at most~1.
In the same paper, Fu introduced the Equitable Total Coloring Conjecture.

\begin{conj}[\cite{F94}] \label{conj: ETCC}
For every graph $G$, $G$ has an equitable total $k$-coloring for each $k \geq \max\{\chi''(G), \Delta(G)+2 \}$.
\end{conj}

The ``$\Delta(G)+2$'' is required because Fu~\cite{F94} found an infinite family of graphs $G$ with $\chi''(G)= \Delta(G)+1$ but no equitable total $(\Delta(G)+1)$-coloring.
Note that if Conjecture~\ref{conj: TCC} is true, we would have $\max\{\chi''(G), \Delta(G)+2 \} = \Delta(G)+2$.

In terms of total graphs, Conjecture~\ref{conj: ETCC} states that for any graph $G$, $T(G)$ has an equitable $k$-coloring for each $k \geq \max\{\chi(T(G)), \Delta(G)+2 \}$.  Since $\Delta(T(G)) = 2 \Delta(G)$, this is much stronger than Theorem~\ref{thm: HS} and the $\Delta$-ECC for total graphs.

Fu~\cite{F94} showed that Conjecture~\ref{conj: ETCC} holds for complete bipartite graphs, complete $t$-partite graphs of odd order, trees, and certain split graphs.  Equitable total coloring has also been studied for graphs with maximum degree 3~\cite{W02}, joins of certain graphs~\cite{GM12,GZ09,ZW05}, the Cartesian product of cycles~\cite{CX09}, and the corona product of cubic graphs~\cite{F15}.

\subsection{Equitable Choosability and List Total Coloring}\label{ec}

List coloring is a well-studied variation of classic vertex coloring, and list versions of total coloring and equitable coloring have each received a lot of attention.
List coloring was introduced independently by Vizing~\cite{V76} and Erd\H{o}s, Rubin, and Taylor~\cite{ET79} in the 1970's.  For a graph $G$, if each vertex $v \in V(G)$ is assigned a list of colors $L(v)$, we say $L$ is a \emph{list assignment} for $G$.  The graph $G$ is \emph{$L$-colorable} if there exists a proper coloring $f$ of $G$ such that $f(v) \in L(v)$ for each $v \in V(G)$ (and $f$ is a \emph{proper $L$-coloring} of $G$).  The \emph{list chromatic number} of a graph $G$, denoted $\chi_l(G)$, is the smallest $k$ such that $G$ is $L$-colorable whenever $L$ is a list assignment with lists of size $k$ (or more).  We say that a graph $G$ is \emph{k-choosable} when $\chi_l(G) \leq k$.  When the lists associated with the list assignment $L$ have uniform size $k$, we say that $L$ is a \emph{$k$-assignment}.

Considering identical lists yields $\chi(G) \leq \chi_l(G)$, while greedy coloring yields $\chi_l(G)\leq \Delta(G)+1$, for all graphs $G$. Vizing~\cite{V76} proved a list version of Brooks' Theorem: every connected graph $G$ which is neither a complete graph nor an odd cycle satisfies $\chi_l(G) \leq \Delta(G)$.  Erd\H{o}s, Taylor, and Rubin~\cite{ET79} observed that bipartite graphs can have arbitrarily large list chromatic numbers, which shows that the gap between $\chi(G)$ and $\chi_l(G)$ can be arbitrarily large.
Graphs for which $\chi(G) = \chi_l(G)$ are called \emph{chromatic-choosable} graphs~\cite{O02}.

List total coloring (i.e., list coloring of total graphs) has been studied by many researchers.  In 1997, Borodin, Kostochka, and Woodall introduced the List Total Coloring Conjecture.

\begin{conj}[\cite{BK97}] \label{conj: LTCC}
For any multigraph $G$, $\chi(T(G)) = \chi_l(T(G))$; i.e., total graphs are chromatic-choosable.
\end{conj}

Conjecture~\ref{conj: LTCC} has been verified for certain planar graphs~\cite{C09,LM13,LX11,WL16} and multicircuits~\cite{KW021,KW022}.  In~\cite{KW01} it was conjectured that the square of every graph is chromatic-choosable (a much stronger conjecture), but this was recently disproved by Kim and Park~\cite{KP15}.

A list analogue of equitable coloring was introduced in 2003 by Kostochka, Pelsmajer, and West~\cite{KP03}.
Suppose $L$ is a $k$-assignment for the graph $G$.  A proper $L$-coloring of $G$ is \emph{equitable} if each color appears on at most $\lceil |V(G)|/k \rceil$ vertices.
A graph is \emph{equitably $k$-choosable} if an equitable $L$-coloring of $G$ exists whenever $L$ is a $k$-assignment for $G$.

While a $k$-choosable graph is always $k$-colorable, it can happen that a graph is equitably $k$-choosable but not equitably $k$-colorable: for example, $K_{1,6}$ with $k=3$.
Like equitable $k$-coloring, equitable $k$-choosability of a graph $G$ is not necessarily inherited by components or subgraphs $H$:
the color class size restriction may be harder to satisfy. Unlike equitable $k$-coloring, equitable $k$-choosability of its
components does not imply that the whole (disconnected) graph is also equitably $k$-choosable: for example, $K_{1,6}+K_2$ with $k=3$.

%We now mention a bit of terminology which will be used frequently throughout this paper.  Suppose that $H$ is a subgraph of $G$.  Also, suppose that $L$ is a $k$-assignment for the graph $G$.  When we say there is an equitable $L$-coloring of $H$, we mean there is a equitable $L'$-coloring of $H$ where $L'$ is simply the $k$-assignment for $H$ defined by $L'(v)=L(v)$ for each $v \in V(H)$.  Note that this requires color classes of size at most $\lceil{|V(H)|/k}\rceil$, which may be more restrictive than the bound required for an equitable $L$-coloring of $G$.  (Compare with ordinary $L$-coloring, which is automatically a valid $L'$-coloring when restricted to the subgraph $H$; it follows that $\chi_l(H)\leq\chi_l(G)$.)

List analogues of Theorem~\ref{thm: HS} and Conjecture~\ref{conj: ECC} are proposed in~\cite{KP03}.

\begin{conj}[\cite{KP03}] \label{conj: KPW1}
Every graph $G$ is equitably $k$-choosable when $k \geq \Delta(G)+1$.
\end{conj}

\begin{conj}[\cite{KP03}] \label{conj: KPW2}
A connected graph $G$ is equitably $k$-choosable for each $k \geq \Delta(G)$ if it is different from $K_m$, $C_{2m+1}$, and $K_{2m+1,2m+1}$.
\end{conj}

Conjectures~\ref{conj: KPW1} and~\ref{conj: KPW2} have been proved for forests, connected interval graphs, 2-degenerate graphs with $\Delta(G)\ge 5$, and small graphs (at most $2k$ vertices)~\cite{KP03}, as well as outerplanar graphs~\cite{ZB10}, series-parallel graphs~\cite{ZW11}, and other classes of planar graphs~\cite{LB09,ZB08,ZB15}).  In 2013, Kierstead and Kostochka made substantial progress on Conjecture~\ref{conj: KPW1}, as follows.

\begin{thm}[\cite{KK13}] \label{thm: KKresult}
If $G$ is any graph, then $G$ is equitably $k$-choosable whenever
\[
k \geq
\begin{cases}
\Delta(G)+1 & \text{if} \; \Delta(G) \leq 7 \\
\Delta(G) + \frac{\Delta(G)+6}{7} & \text{if } \; 8 \leq \Delta(G) \leq 30 \\
\Delta(G) +   \frac{\Delta(G)}{6} & \text{if } \; \Delta(G) \geq 31.
\end{cases}
\]
\end{thm}

We are finally ready to turn our attention to total equitable choosability, a combination originally suggested by Nakprasit~\cite{N02}.
We begin with a natural extension of Conjecture~\ref{conj: ETCC}, which may be called the List Equitable Total Coloring Conjecture.

\begin{conj} \label{conj: LETCC}
For every graph $G$, $T(G)$ is equitably $k$-choosable for each \\ $k \geq \max \{\chi_l(T(G)), \Delta(G)+2 \}$.
\end{conj}

Fu's infinite family of graphs $G$ with $\chi''(G)= \Delta(G)+1$ and no equitable total $(\Delta(G)+1)$-coloring also has the property that $T(G)$ is not equitably $(\Delta(G)+1)$-choosable, so the Conjecture~\ref{conj: LETCC} would be sharp if true.
Also note that since $\Delta(T(G)) = 2 \Delta(G)$, the Conjecture~\ref{conj: LETCC} is saying something stronger about total graphs than Conjectures~\ref{conj: KPW1} and~\ref{conj: KPW2}.

\subsection{Results}

For our main result, we will need to consider certain powers of paths and powers of cycles. Recall that $G^k$, power of a graph $G$, has the same vertex set as $G$ and edges between any two vertices within distance $k$ in $G$.
Since powers of paths are interval graphs, Conjectures~\ref{conj: KPW1} and~\ref{conj: KPW2} follow from~\cite{KP03},
but a stronger result is possible.
%The maximum degree of $P_n^p$ is $\min\{n-1,2p\}$, so the bound in~Theorem~\ref{thm: pathpowers} is stronger than what was previously shown.
%Theorem~\ref{thm: pathpowers} also easily extends to disconnected graphs where every component is a power of a path (Corollary~\ref{cor: pathpowers}); recall from Section~\ref{ec} that this doesn't go without saying.

\begin{thm} \label{thm: pathpowers}
For $p, n \in \N$, $P_n^p$ is equitably $k$-choosable whenever $k \geq p+1$.
\end{thm}

Note that $P_n^p$ is a complete graph when $n \leq p+1$ and $C_n^p$ is a complete graph when $n \leq 2p+1$.
Theorem~\ref{thm: pathpowers} is sharp because $P_n^p$ contains a copy of $K_{p+1}$ (unless $n\le p$)
so it isn't even $p$-colorable.

We next prove Conjectures~\ref{conj: KPW1} and~\ref{conj: KPW2} for powers of cycles.
Note that Theorem~\ref{thm: cyclepowers} needs $p \geq 2$ to avoid odd cycles and it needs $n \geq 2p+2$ to avoid complete graphs.

\begin{thm} \label{thm: cyclepowers}
For $p , n \in \N$ with $p \geq 2$ and $n \geq 2p+2$, $C_{n}^p$ is equitably $k$-choosable for each $k \geq 2p = \Delta(C_n^p)$.
\end{thm}

We will use a lemma from~\cite{KP03} which was used to prove many earlier results.
If $L$ is a list assignment for a graph $G$ and $H$ is a subgraph of $G$, we will also consider $L$ to be a list assignment for $H$ by restricting $L$ to the vertices of $H$.
Let $N_G(v)$ represent the \emph{neighborhood of $v$}, which is the set of vertices in $V(G)$ adjacent to $v$ in $G$.

\begin{lem}[\cite{KP03}] \label{lem: pelsmajer} Let $G$ be a graph, and let $L$ be a $k$-assignment for $G$.  Let \\ $S=\{x_1, x_2, \ldots, x_k \}$ be a set of $k$ vertices in $G$.  If $G - S$ has an equitable $L$-coloring and
$$|N_{G}(x_i) - S| \leq k-i$$
for $1 \leq i \leq k$, then $G$ has an equitable $L$-coloring.
\end{lem}

Lemma~\ref{lem: pelsmajer} is all that is needed for the proof of Theorem~\ref{thm: pathpowers}.  However, for our proof of Theorem~\ref{thm: cyclepowers}, we have created a slightly more general lemma which is a bit trickier to apply.

For any graph $G$ and $U \subseteq V(G)$, let $G[U]$ denote \emph{the subgraph of $G$ induced by $U$}, which is the graph with vertex set $U$ and edge set $\{uv\in E(G): u\in U,v\in U\}$.

\begin{lem} \label{lem: gpelsmajer}
Let $G$ be a graph, and let $L$ be a $k$-assignment for $G$.  Let $S=\{x_1, x_2, \ldots, x_{mk} \}$ where $m \in \N$ and $x_1, x_2, \ldots, x_{mk}$ are distinct vertices in $G$.  Suppose that $c$ is an equitable $L$-coloring of $G-S$.  For each $i$ satisfying $1 \leq i \leq mk$, let

$$ L'(x_i) = L(x_i) - \{c(u) : u \in N_G(x_i) - S \}.$$

\noindent If there is a proper $L'$-coloring of $G[S]$ which has no color class of size exceeding $m$, then $G$ has an equitable $L$-coloring.
\end{lem}

That all happens in Section~\ref{cyclepowers}.  In Section~\ref{mainresult} we
prove our main result: we verify Conjecture~\ref{conj: LETCC} for graphs $G$ with $\Delta(G)\leq 2$.

\begin{thm}\label{thm: totalmain}
If $G$ is a multigraph with $\Delta(G) \leq 2$, then $T(G)$ is equitably $k$-choosable for each $k \geq \Delta(G)+2$.  In particular, the List Equitable Total Coloring Conjecture holds for all graphs $G$ with $\Delta(G) \leq 2$.
\end{thm}

If $G$ is a graph with $\Delta(G) \leq 2$, then its components are paths and cycles.  Recall that $T(P_m)= P_{2m-1}^2$ and $T(C_n) = C_{2n}^2$. Thus, by the end of Section~\ref{cyclepowers} we will already have shown Theorem~\ref{thm: totalmain} for connected graphs and linear forests, but that does not suffice:
recall that in general, it isn't enough to prove equitable $k$-choosability for every component of a graph.
The disconnected case of the proof of Theorem~\ref{thm: totalmain} will require heavy use of Lemma~\ref{lem: gpelsmajer} and further ideas.
We should note that Theorem~\ref{thm: totalmain} is obvious when $\Delta(G)$ is $0$ or $1$ and the $k\ge 5$ case follows from Theorem~\ref{thm: KKresult}, so we need only consider the case where $\Delta(G)=2$ and $k=4$.

We will actually prove something a bit stronger than Theorem~\ref{thm: totalmain} in Section~\ref{mainresult}.
Note that for multigraphs $G$ with $\Delta(G)\le 2$, each component of $T(G)$ is a square of an odd path, a square of an even cycle on at least 6 vertices, or a copy of $K_4$.  We will prove equitable $4$-choosability for graphs where components may be square of \emph{any} path, square of \emph{any} cycle on at least 6 vertices, or a copy of $K_4$, and we also prove equitable $3$-choosability for graphs where each component is a square of a cycle of length divisible by~3 or a square of a path (Theorem~\ref{lem: mainlemma}). So we find exactly which total graphs $T(G)$ with $\Delta(G)=2$ are equitably 3-choosable.

Finally, we mention that Conjecture~\ref{conj: LETCC} has also been proved for trees of maximum degree 3, stars, and double stars~\cite{M18}, but for the sake of brevity, we have not included these results in this paper.

\section{Equitable Choosability of Powers of Paths and Cycles} \label{cyclepowers}

List coloring powers of cycles and paths is well understood.  It is easy to see that $\chi_l(P_n^p) = p+1$ whenever $1 \leq p \leq n-1$.  Prowse and Woodall~\cite{PW03} determined the chromatic number for all cycle powers and showed that powers of cycles are chromatic-choosable.

We begin with a straightforward application of Lemma~\ref{lem: pelsmajer}.

%\begin{customthm} {\bf \ref{thm: pathpowers}}
%Suppose that $n, p \in \N$.  Then, $P_n^p$ is equitably $k$-choosable whenever $k \geq p+1$.
%\end{customthm}

\begin{proof}[Proof of Theorem~\ref{thm: pathpowers}]
We will use induction on $n$, with $p$ and $k$ fixed.
Let $L$ be any $k$-assignment for $G=P_n^p$.
If $n\leq k$, it suffices to greedily $L$-color vertices with distinct colors.

Suppose that $n > k$ and that the desired result holds for natural numbers less than $n$.  Label the vertices $v_1, v_2, \ldots, v_n$ in order taken from the underlying path $P_n$.  Now, let $x_i = v_{k-i+1}$ for each $1 \leq i \leq k$.  We let $S=\{x_1, x_2, \ldots, x_k \}$.  Note that $|N_{G}(x_i) - S|=\max\{0,p-i+1\}\leq k-i$ for each $x_i\in S$.
The inductive hypothesis guarantees that $G-S$ has an equitable $L$-coloring.  Applying
Lemma~\ref{lem: pelsmajer} completes the proof.
\end{proof}

%As discussed in~Section~\ref{ec}, equitable $k$-choosability on components doesn't usually imply equitable $k$-choosability of the whole graph.  However in this case, it works.
%
%\begin{cor} \label{cor: pathpowers}
%Suppose that $G$ is a graph whose components $P_{n_1}^{p_1}, \ldots, P_{n_\ell}^{p_\ell}$ are all powers of paths.
%Then $G$ is equitably $k$-choosable for each $k \geq 1+\max\{p_1, \ldots, p_\ell\}$.
%\end{cor}
%
%\begin{proof}
%Observe that $G$ is a spanning subgraph of a copy of $P_n^p$, where $n = \sum_{i=1}^\ell n_i$ and $p=\{p_1, \ldots, p_\ell\}$.
%Given any $k$-assignment for $G$ with $k\ge p+1$, add edges to obtain $P_n^p$, then apply Theorem~\ref{thm: pathpowers}.
%\end{proof}

Next, we work toward proving Theorem~\ref{thm: cyclepowers}. We begin with the proof of Lemma~\ref{lem: gpelsmajer}.

\begin{proof}[Proof of Lemma~\ref{lem: gpelsmajer}]
Let $n=|V(G)|$.  Then each color class associated with $c$ has size at most $\lceil (n-mk)/k \rceil = \lceil n/k \rceil - m$.  Suppose that $c'$ is a proper $L'$-coloring of $G[S]$ which uses no color more than $m$ times.  If we combine $c$ and $c'$, we get a proper $L$-coloring of $G$ such that each color class has size at most $\left \lceil n/k \right \rceil$, as required.
%: $\left \lceil n/k \right \rceil - m + m = \left \lceil n/k \right \rceil$.
%Thus, $G$ has an equitable $L$-coloring.
\end{proof}

The next lemma is a straightforward fact.  We use $\alpha(G)$ to denote the \emph{independence number} of $G$, which is the size of the largest independent set of vertices in $G$.

\begin{lem} \label{lem: independent}
Suppose that $G$ is a graph with $n$ vertices, and suppose $k$ is a positive integer such that $\alpha(G) \leq  \lceil n/k \rceil$ and $\chi_l(G) \leq k$.  Then, $G$ is equitably $k$-choosable.
\end{lem}

\begin{proof}
Suppose that $L$ is an any $k$-assignment for $G$.  Since $\chi_l(G) \leq k$, we know that there is a proper $L$-coloring, $c$, of $G$.  Moreover, since each of the color classes associated with $c$ are independent sets, we know that there is no color classes associated with $c$ of size exceeding $\alpha(G)$.  Since $\alpha(G) \leq  \lceil n/k \rceil$, we have that $c$ is an equitable $L$-coloring of $G$.
\end{proof}

Now, we extend the basic idea of Lemma~\ref{lem: pelsmajer} to apply when the bound $|N_G(x_i) - S|\leq k-i$ isn't quite satisfied.
To do that, it's not enough to look at the \emph{number} of colors remaining for each $x_i\in S$ after removing colors on neighbors
not in $S$; we need to actually look at \emph{which colors} remain available for each $x_i\in S$.

\begin{lem} \label{lem: cyclepowers}
Suppose $H$ is a graph with $V(H) = \{x_i : 1 \leq i \leq k \}$.  Suppose $L$ is a $k$-assignment for $H$ such that $|L(x_i)| \geq i$ for $1 \leq i \leq k-1$ and $|L(x_k)| \geq k-1$.  If $|L(x_k)| \geq k$ or $L(x_{k-1}) \neq L(x_k)$, then $H$ is $L$-colorable with $k$ distinct colors.
\end{lem}

\begin{proof}
The result is obvious when $|L(x_k)| \geq k$.  So, suppose that $|L(x_k)| = k-1$ and $L(x_{k-1}) \neq L(x_k)$.  By the restrictions on the list sizes, we can greedily color the vertices $x_1, x_2, \ldots, x_{k-2}$ with $k-2$ distinct colors.  Let $C$ be the set of $k-2$ colors used to color these vertices, and let $L'(v_i) = L(v_i)-C$ for $i=k-1, k$.  Note that $|L'(v_i)| \geq 1$ for $i=k-1, k$, and we are done if $L'(v_{k-1})$ or $L'(v_k)$ have two or more elements.  So, assume $|L'(v_i)| = 1$ for $i=k-1,k$.  Then, $C \subset L(v_i)$ for $i = k-1, k$, and since $L(x_{k-1}) \neq L(x_k)$, we know $L'(x_{k-1}) \neq L'(x_k)$.  Thus, we can complete a proper $L$-coloring of $H$ with $k$ distinct colors.
\end{proof}

We now prove Theorem~\ref{thm: cyclepowers}.
%, restated below.
%which verifies Conjectures~\ref{conj: KPW1} and~\ref{conj: KPW2} for powers of cycles.

%\begin{customthm} {\bf \ref{thm: cyclepowers}}
%Suppose that $p , n \in \N$ with $p \geq 2$ and $n \geq 2p+2$.  Then, $C_{n}^p$ is equitably $k$-choosable for each $k \geq 2p = \Delta(C_n^p)$.
%\end{customthm}

\begin{proof}[Proof of Theorem~\ref{thm: cyclepowers}]
Let $G=C_n^p$, and $V(G) = \{v_1, v_2, \ldots, v_n \}$ where the vertices are written in cyclic order based upon the underlying $C_n$ used to form $G$. Let $L$ be an arbitrary $k$-assignment for $G$.  If $k \geq n$, we can obtain an equitable $L$-coloring of $G$ by coloring the vertices of $G$ with $n$ distinct colors.  Thus, we may assume that $n \geq k+1$.

\par

If $k \geq 2p+1$, we will apply Lemma~\ref{lem: pelsmajer}.  Let $S = \{v_i : 1 \leq i \leq k \}$.  Then, for $1 \leq i \leq p+1$, $|N_G(v_i) - S| \leq p -i+1$.  For $p+2 \leq i \leq k$, $|N_G(v_i) - S| \leq p$.  Now, let $x_i = v_{k+1-i}$ for $1 \leq i \leq k-p-1$, and let $x_{k-p} = v_1, x_{k-p-1} = v_2, \ldots, x_k = v_{p+1}$.  It is now easy to check $|N_G(x_i) - S| \leq k - i$ for $1 \leq i \leq k$.  Since $G-S$ is the square of a path, Theorem~\ref{thm: pathpowers} implies $G-S$ has an equitable $L$-coloring.  Lemma~\ref{lem: pelsmajer} then implies that $G$ has an equitable $L$-coloring.  So, we assume that $k=2p$.

\par

Suppose that $n \geq 2p+3$.  Let $S = \{v_i : 1 \leq i \leq 2p \}$, $S' = \{v_i : 2p+3 \leq i \leq \min\{4p,n\} \}$, and $S'' = \{v_{2p+1}, v_{2p+2} \}$.  We let $A = S \cup S' \cup S''$.  Note that if $n \geq 4p+1$, $H = G - A$ is the square of a path and Theorem~\ref{thm: pathpowers} implies there is an equitable $L$-coloring of $H$.  Suppose we color $H$ according to an equitable $L$-coloring.  With the intent of using Lemma~\ref{lem: gpelsmajer}, we will show that we can find a proper $L$-coloring of $G[A]$ which uses no color more than twice, and is not in conflict with the coloring used for $H$.

\par

We claim we can greedily $L$-color the vertices in $S'$ in reverse order with distinct colors while avoiding colors already used on neighbors of each vertex.  Specifically, the $i$th vertex in this order for $1 \leq i \leq p$ must avoid $i-1$ colors already used on $S'$ and up to $p-i+1$ colors already used on $V(H)$, and there are $k > p = (i-1)+(p-i+1)$ colors choices available.  Each remaining vertex of $S'$ must only avoid the colors already used on vertices in $S'$ which is doable since $k > 2p-2 \geq |S'|$.  Thus, $S'$ is colored with distinct colors, and we call the proper $L$-coloring of $G - (S \cup S'')$ we have thus far $c$.

\par

Note that $v_p$ has exactly one neighbor that is already colored, $v_n$.  Let $L'(v_p) = L(v_p) - \{c(v_n) \}$.  If $|L'(v_p)| = 2p-1$, we can let $c_{p+1}^*$ be an element in $L(v_{p+1}) - L'(v_p)$.  Otherwise, let $c_{p+1}^*$ be any element of $L(v_{p+1})$ (note we have not colored $v_{p+1}$ yet).

\par

We will now color the vertices in $S''$.  Clearly, all their neighbors that have been colored are in $S'$ since $2p+2+p \leq 4p$.  Since $|S'| \leq 2p-2$ and $|L(v_{2p+1})| \geq 2p$, we can color $v_{2p+1}$ with a color in $L(v_{2p+1})$ which avoids the colors used on $S'$ and avoids the color $c_{p+1}^*$.  Then, we can color $v_{2p+2}$ with a color in $L(v_{2p+2})$ which avoids the colors used on $S' \cup \{v_{2p+1} \}$.  We have now colored $S' \cup S''$ with $2p$ distinct colors.

\par

Now, for each $v_i \in S$ let $L'(v_i)$ be equal to $L(v_i)$ after removing all colors used on neighbors of $v_i$.  Notice that $c_{p+1}^* \in L'(v_{p+1})$.  Also, for $1 \leq i \leq p$, $|L'(v_i)| \geq p+i-1$ and $|L'(v_{2p-i+1})| \geq p+i-1$.  With the intent of applying Lemma~\ref{lem: cyclepowers}, we rename the vertices of $S'$ as follows
$$x_1=v_{2p}, x_2 = v_1, x_3 = v_{2p-1}, x_4= v_2, \ldots, x_{2p-1}=v_{p+1}, x_{2p}= v_p.$$
If $|L'(v_p)|=2p-1$, then $c_{p+1}^* \in L'(v_{p+1}) - L'(v_p)$.  So, $L'(v_p) \neq L'(v_{p+1})$.  Thus, Lemma~\ref{lem: cyclepowers} applies, and we can find a proper $L'$-coloring of $G[S]$ which uses $2p$ distinct colors.  By Lemma~\ref{lem: gpelsmajer}, the case where $n \geq 2p+3$ is complete.

\par

Finally, when $n =2p+2$, since $G$ is not a complete graph or odd cycle, $\chi_l(G) \leq \Delta(G) = 2p$ by the Vizing's extension of Brooks' Theorem~\cite{V76}.  Also, $ \alpha(G) = \lfloor n/(p+1) \rfloor = 2 = \lceil n/(2p) \rceil$. So, Lemma~\ref{lem: independent} implies that there is an equitable $L$-coloring of $G$. (Note that this last argument actually works whenever $2p+2 \leq n \leq 3p+2$ and $k=2p$.)
\end{proof}

\section{Maximum Degree Two} \label{mainresult}

In this section we prove our main result, Theorem~\ref{thm: totalmain}.  The difficulty of the proof is that we must deal with disconnected graphs.  We heavily rely on Theorem~\ref{thm: KKresult} and Lemma~\ref{lem: gpelsmajer} to prove this result, which we restate here.

\begin{customthm}{\bf \ref{thm: totalmain}}
If $G$ is a multigraph with $\Delta(G) \leq 2$, then $T(G)$ is equitably $k$-choosable for each $k \geq \Delta(G)+2$.  In particular, the List Equitable Total Coloring Conjecture holds for all graphs $G$ with $\Delta(G) \leq 2$.
\end{customthm}

Theorem~\ref{thm: totalmain} is obvious when $\Delta(G)=0$.  When $\Delta(G) = 1$, the graph $T(G)$ consists of disjoint copies of $K_3$ and isolated vertices (and there must be at least one copy of $K_3$), so $\Delta(T(G))=2$ and Theorem~\ref{thm: KKresult} implies that $T(G)$ is equitably $k$-choosable for each $k \geq 3 = \Delta(G)+2$ (although one could instead give an inductive proof using Lemma~\ref{lem: pelsmajer}).

\par

When $G$ is a multigraph (or simple graph) with $\Delta(G)=2$, proving Theorem~\ref{thm: totalmain} is not as straightforward. In this case $T(G)$ consists of the disjoint union of squares of cycles on at least 6 vertices, squares of paths, and copies of $K_4$.  Theorem~\ref{thm: KKresult} tells us that $T(G)$ is equitably $k$-choosable for each $k \geq \Delta(T(G))+1$, but $\Delta(T(G))+1$ may be as large as 5.  So, we still need to show that $T(G)$ is equitably 4-choosable in this case.  From the previous section, we know that squares of cycles of order at least 6, squares of paths, and copies of $K_4$ are all equitably 4-choosable.  However, one should recall from Section~\ref{ec} that this does not necessarily imply that the disjoint union of such graphs will be equitably $4$-choosable.  Overcoming this obstacle is the main difficulty in the proof of Theorem~\ref{thm: totalmain}.

\par

Theorem~\ref{lem: mainlemma} will complete the proof of Theorem~\ref{thm: totalmain}. In the proof we will use the fact from Prowse and Woodall~\cite{PW03} that when $n \geq 6$, $\chi_l(C_n^2)=3$ if and only if 3 divides $n$ and $\chi_l(C_n^2)=4$ otherwise.

\begin{thm} \label{lem: mainlemma}
Suppose that $G$ is a graph with components that are squares of paths, squares of cycles on at least 6 vertices, and copies of $K_4$.  Then $G$ is equitably $4$-choosable.  Moreover, $G$ is equitably 3-choosable if and only if its components consist of no copies of $K_4$ and all of its components that are squares of cycles have order divisible by 3.

\end{thm}

Theorem~\ref{lem: mainlemma} is stronger than Theorem~\ref{thm: totalmain} in a couple ways.  One, it includes many graphs that are not total graphs.  Two, it tells us exactly which total graphs $T(G)$ with $\Delta(G)=2$ are equitably 3-choosable.

\begin{proof}
If there are any components that are squares of paths, let $Q$ be their union.
Since the underlying paths form a spanning subgraph of a single path, we can add edges to obtain the square of a path.
Adding edges can only make the problem more difficult, so we may assume that $Q$ is a square of a path.
Also, let $R_1,\ldots,R_m$ be the squares of cycles.

We will prove the second part first.  If $G$ is equitably 3-choosable, then each of its components are 3-choosable.  Since $\chi_l(K_4)=4$ and $\chi_l(C_n^2)=4$ when 3 does not divide $n$, we know that the components of $G$ consist of no copies of $K_4$, and all of $G$'s components that are squares of cycles have order divisible by 3.

Conversely, suppose each $|V(R_i)|$ is divisible by~3 and $G=Q+\sum_{i=1}^m R_i$.
Let $L$ be an arbitrary 3-assignment for $G$.  Since $\chi_l(R_i)=3$ and $\alpha(R_i)=|V(R_i)|/3$,
Lemma~\ref{lem: independent} yields a proper $L$-coloring of $R_i$ with color classes of size at most $|V(R_i)|/3$,
for each $i$. By Theorem~\ref{thm: pathpowers}, $Q$ has a proper $L$-coloring with color classes of size at most $\lceil |V(Q)|/3 \rceil$.
Combining these colorings gives a proper $L$-coloring of $G$ with color classes of size at most
$$\lceil |V(Q)|/3 \rceil +\sum_{i=1}^m |V(R_i)|/3 = \left\lceil{ |V(Q)|/3+\sum_{i=1}^m |V(R_i)|/3}\right\rceil= \left\lceil{|V(G)/3}\right\rceil$$
since $|V(R_i)|/3$ is an integer for each $i$.  Thus, $G$ is equitably $3$-choosable.

Now we prove the first part by induction on $|V(G)|$.  The result is clear when $|V(G)| \leq 4$.  So, suppose $|V(G)| \geq 5$ and the result holds for graphs on fewer than $|V(G)|$ vertices.  Let $L$ be any $4$-assignment for $G$.
If $|V(Q)|\ge 4$, delete four vertices from one end, and apply induction to $L$-color the rest of the graph.  Then apply Lemma~\ref{lem: pelsmajer} to extend the $L$-coloring to $G$.
If there are any copies of $K_4$ in $G$, remove one and apply induction, then apply Lemma~\ref{lem: pelsmajer}.
So, assume that there are no copies of $K_4$ and that $|V(Q)|\le 3$.
Since $|V(G)| \geq 5$, $m\geq 1$.

Suppose without loss of generality that $|V(R_1)|=\max_i |V(R_i)|$. Let $n=|V(R_i)|$ and label the vertices of $R_1$ in cyclic order $v_1,\ldots,v_n$.  If $n\geq 8$, then we can apply the method used in the proof of Theorem~\ref{thm: cyclepowers} (with $p=2$ and $k=4$ and this $n$): delete a set $S= \{v_1, \ldots, v_8 \}$, apply induction to $G-S$, then use Lemmas~\ref{lem: gpelsmajer} and~\ref{lem: cyclepowers} to extend to $S$.

It remains to consider cases $n=6$ and $n=7$.

\textbf{Case where $n=6$}:  First, we will specify a proper $L$-coloring of $R_1$.
Pick any color $c_1\in L(v_1)$ and let $c(v_1)=c_1$.  Let $L'(v_i)=L(v_i)- \{c_1 \}$ for $2\le i\le 6$.  Since $\chi_l(R_1)=3$, we can find a proper $L'$-coloring of $R_1 - \{v_1\}$ to complete $c$.  Since $\alpha(R_1 - \{v_1\}) = 2$, at most two color classes associated with $c$ have size~2, and the rest of the color classes are of size~1. Let $C$ be the set of colors used twice by $c$ on $R_1$.

Now consider the sub-case where $m \geq 2$. In this case, $|V(R_2)|=6$. Let $u_1, u_2, \ldots, u_6$ be the vertices of $R_2$ in cyclic order.  Note that the only independent sets of size greater than~1 in $R_2$ are $\{u_1,u_4\}$, $\{u_2, u_5\}$, and $\{u_3,u_6\}$.

Now, for each $1\le i\le 3$:
If $L(u_i) \cap C \neq \emptyset$, let $L'(u_i)$ be obtained from $L(u_i)$ by deleting one of the elements in $L(u_i) \cap C$; otherwise, let $L'(u_i)=L(u_i)$.  Next, if $L'(u_i)\cap C \neq \emptyset$ and $L(u_{i+3})$ contains the color in $L'(u_i)\cap C$, remove it from $L(u_{i+3})$ to define $L'(u_{i+3})$; otherwise let $L'(u_{i+3})=L(u_{i+3})$.

By the construction, any proper $L'$-coloring of $R_2$ uses each color in $C$ at most once.  Since $\chi_l(R_2)=3$ and every list $L'(u_i)$ on $R_2$ has size at least~3, there is a proper $L'$-coloring of $R_2$; since $\alpha(R_2)=2$ each color is used at most~2 times on $R_2$.  Combining this with $c$ on $R_1$ we get a proper $L$-coloring of $R_1+R_2$ which uses each color at most~3 times.  If $|V(G)|=6$ we are done; otherwise apply induction and Lemma~\ref{lem: gpelsmajer} to extend the coloring to $G$.

Now consider the sub-case $m=1$.  If $|V(G)|\le 8$ we need a proper $L$-coloring with color classes of size at most~2.  We are done if $G=R_1$.  If $G=R_1+Q$ and $|V(Q)|$ is~1 or~2, we can remove $C$ from lists on $V(Q)$ and then $L$-color $Q$ without repeating a color, which suffices.  Otherwise $|V(Q)|=3$, in which case we just greedily $L$-color $Q$ with distinct colors.  Then each color is used at most~3 times on $G$, which suffices since $\lceil 9/4 \rceil=3$.

\textbf{Case where $n=7$}:  Suppose that $n=7$.  First consider the sub-case where $m \geq 2$.  In this sub-case $6 \leq |V(R_2)| \leq 7$, and suppose that $u_1, u_2, \ldots, u_s$ ($s$ is 6 or 7) are the vertices of $R_2$ in cyclic order.  With the intent of applying Lemma~\ref{lem: gpelsmajer}, let $S = \{v_1, v_2, \ldots, v_7, u_1, u_2, \ldots, u_5 \}$.  By induction there is an equitable $L$-coloring, $c$, of $G-S$, and we let $L'$ be the list assignment for $G[S]$ given by $L'(x) = L(x) - \{c(u) : u \in (N_G(x) - S) \}$ for each $x \in S$.  Now, consider the graph $H=G[u_1, u_2, \ldots, u_5]$.  We claim that there exists a proper $L'$-coloring, $c'$, of $H$ which has at most one color class of size~2 and no color classes of size more than~2.  We will prove this claim when $s=7$ and when $s=6$.

\par

When $s=7$, $H$ is a copy of $P_5^2$.  Note that $|L'(u_i)| \geq 2$ for $i=1,5$, $|L'(u_i)| \geq 3$ for $i=2,4$, and $|L'(u_3)| = 4$ (since $L(u_3)=L'(u_3)$).  We form $c'$ by first coloring $u_4$ with $c_4 \in L'(u_4)$ so that $|L'(u_5) - \{c_4 \}| \geq 2$ (this is possible since $|L'(u_4)| \geq 3$ and $|L'(u_5)| \geq 2$).  Now, let $L''(u_i) = L'(u_i) - \{c_4 \}$ for $i=1,2,3,5$.  For $i=1,2,3$, $|L''(u_i)| \geq i$.  This means we may greedily color $u_1, u_2$, and $u_3$ with distinct colors from $L''$.  Having colored $u_1, u_2, u_3, u_4$ with~4 distinct colors, we may complete $c'$ by coloring $u_5$ with a color in $L''(u_5)$ distinct from the color used on $u_3$.

\par

When $s=6$, note that $H$ is the same as in the case $s=7$ except $u_1$ is adjacent to $u_5$. When $s=6$, $|L'(u_i)| \geq 3$ for $i=1,2,4,5$ and $|L'(u_3)| = 4$ (since $L(u_3)=L'(u_3)$).  We form $c'$ by coloring $u_3$ with $c_3 \in L'(u_3)$ so that $|L'(u_4) - \{c_3 \}| \geq 3$ (this is possible since $|L'(u_3)|=4$ and $|L'(u_4)| \geq 3$).  Now, let $L''(u_i) = L'(u_i) - \{c_3 \}$ for $i=1,2,4,5$.  We notice that $|L''(u_i)| \geq 2$ for $i=1,2,5$ and $|L''(u_4)| \geq 3$.  Also, $H[\{u_1, u_2, u_4, u_5\}]$ is a 4-cycle.  We complete $c'$ by finding a proper $L''$-coloring, $f$, of this~4 cycle.  If $f$ has two color classes of size~2, we recolor $u_4$ with a color not used by $f$ (this is possible since $|L''(u_4)| \geq 3$).  The resulting coloring has the desired property.

\par

Having proven our claim, we now return to the sub-case where $n=7$ and $m \geq 2$.  In order to apply Lemma~\ref{lem: gpelsmajer}, we will show that there is a proper $L'$-coloring of $G[S]$ which has no color class of size more than~3.  To construct such a coloring, we begin with a proper $L'$-coloring, $c'$, of $H$ which has at most one color class of size~2 and no color class of size more than~2.  Suppose that $c_1$ is the color used twice by $c'$, or an arbitrary color used by $c'$ if $c'$ uses no color twice.  Notice that $L(v_i)=L'(v_i)$ for $i=1,2, \ldots 7$.  If $c_1 \notin \cup_{i=1}^7 L'(v_i)$, we find a proper $L'$-coloring, $c''$, of $R_1$ (which is possible since $\chi_l(R_1)=4$).  Note $c''$ never uses the color $c_1$ and no color class associated with $c''$ has size more than $\alpha(R_1) = \lfloor 7/3 \rfloor = 2$.  Thus, if we combine $c'$ and $c''$, we obtain a proper $L'$-coloring of $G[S]$ with no color class of size more than~3.  So, without loss of generality, suppose that $c_1 \in L'(v_1)$.  In this situation we color $v_1$ with $c_1$, and we let $L''(v_i)=L'(v_i)- \{c_1\}$ for $i=2,3, \ldots, 7$.  We note that $R_1 - \{v_1\}$ is a spanning subgraph of a copy of $C_6^2$.  Since $\chi_l(C_6^2)=3$, there exists a proper $L''$-coloring of $R_1 - \{v_1\}$ which uses no color more than $\alpha(C_6^2)=2$ times.  Let $c''$ be the coloring of $R_1$ obtained when we color $v_1$ with $c_1$ and use such an $L''$-coloring of $R_1 - \{v_1\}$ for what remains.  If we combine $c'$ and $c''$, we obtain a proper $L'$-coloring of $G[S]$ with no color class of size more than~3 (since $c_1$ is only used once by $c''$) as desired.

\par

Now, consider the sub-case $m=1$.  Let $c'$ be a proper $L$-coloring of $R_1$, and suppose we color $R_1$ according to $c'$.  Clearly, $c'$ has at most~3 color classes of size~2 and no color class of size more than~2.   Let $C$ consist of the colors that are used more than once by $c'$.  If $|V(G)| \leq 8$, we need a proper $L$-coloring with color classes of size at most~2.  We are done if $G= R_1$.  If $G=R_1+Q$ with $|V(Q)|=1$, we color the remaining vertex with a color that is not in $C$.  Otherwise $|V(Q)|$ is~2 or~3, and we use distinct colors to color the remaining vertices, and each color is used at most 3 times on $G$, which suffices since $\lceil 9/4 \rceil = 3$.
\end{proof}

Notice that if our goal was to only prove Theorem~\ref{thm: totalmain} the above proof could be significantly shortened.  In particular, we would not need to prove the result about equitable 3-choosability, and we could completely eliminate the case where $n=7$ since the total graph of a multigraph with maximum degree 2 does not have any squares of odd cycles as components.  The reason we present the stronger result is that the equitable choosability of the disjoint union of powers of paths and cycles is an interesting topic in its own right.

\end{document}